\documentclass[12pt, reqno]{amsart}
\usepackage{amsmath, amsthm, amscd, amsfonts, amssymb, graphicx, color, mathrsfs, commath, stmaryrd}
\usepackage[bookmarksnumbered,plainpages]{hyperref}
\usepackage[all]{xy}

\newtheorem{theorem}{Theorem}[section]
\newtheorem{lemma}[theorem]{Lemma}

\usepackage{enumerate}
\theoremstyle{definition}
\newtheorem{definition}[theorem]{Definition}
\newtheorem{example}[theorem]{Example}

\newtheorem{Theorem}{\quad Theorem}[section]

\newtheorem{remark}[Theorem]{Remark}
\numberwithin{equation}{section}
\begin{document}

\title[A new type of Numerical radius of operators  on Hilbert $C^*$-module]{A new type of Numerical radius of operators on Hilbert $C^*$-module }
\author[M. Mehrazin, M. Amyari and M. E. Omidvar]{Marzieh Mehrazin, Maryam Amyari and Mohsen Erfanian Omidvar$^*$}
\address{Department of Mathematics\\
 Mashhad Branch, Islamic Azad University, Mashhad, Iran}
\email{\tt marzie$\_$mehrazin@yahoo.com} 
\email {\tt  maryam$\_$amyari@yahoo.com;\,\,\, \tt amyari@mshdiau.ac.ir } 
\email{\tt mn$\_$erfanian@yahoo.com}
\keywords{Hilbert $\mathscr{A}$ -module; numerical range; numerical radius.\\
$*$Corresponding author}

\begin{abstract}
In this paper, we define a new concept of numerical range $W_{o}(\cdot)$
and prove its basic results. We also define the numerical radius $\omega_{o}(\cdot)$ and prove that
$$\omega_{o}(T)\leq\interleave T\interleave\leq 2\omega_{o}(T).$$
\end{abstract}
\maketitle
\section{Introduction and preliminaries}

Suppose that  $B(\mathcal{H})$ is the set of all bounded linear operators on a complex Hilbert space $\mathcal{H}$
equipped with the operator norm $\|\cdot\|$.
 The numerical range and the numerical radius are defined by
 $$W(T)=\{\langle Tx,x\rangle :x\in\mathcal{H}, \|x\|=1\},$$ and
  $$\omega(T)=\sup\{|\langle Tx,x\rangle|:x\in\mathcal{H}, \|x\|=1\},$$respectively.
In fact, $\omega(.)$ defines a norm on $B(\mathcal{H})$.\\
It is known that
$$\|T\|=\sup\{|\langle Tx,y\rangle|:x,y\in\mathcal{H}, \|x\|=\|y\|=1\}$$
for each $T\in B(\mathcal{H})$, see \cite[theorem 2.4.1]{Kad}

\indent
If $T\in B(\mathcal{H})$ is a self-adjoint operator, then
 \begin{equation}\label{self adjoint}
 \|T\|=\sup\{|\langle Tx,x\rangle|: x\in\mathcal{H}, \|x\|=1\}
\end{equation}
see \cite[Theorem 4.4.14]{Bon}. In this case $\|T\|=\omega(T)$.\\
\indent
Gustafson \cite[theorem 1.3.1]{Gus} showed that
\begin{equation}\label{nor}
\omega(T)\leq\| T \|\leq 2\omega(T).
\end{equation}
This result show that $\Vert .\Vert$ and $\omega(.)$ are equivalent.\\
\indent
 By using \eqref{self adjoint}, Kittaneh \cite[Theorem 1]{Kit} proved that
\begin{equation}\label{numerical}
\frac{1}{4}\|T^*T+TT^*\|\leq(\omega(T))^2\leq\frac{1}{2}\|T^*T+TT^*\|.
 \end{equation}
 There are several numerical inequalities in the literatur related to inequalities above, see e.g. \cite{Bakh, Dra, Drag, Gol, Sat, Ji, Zam}.

\indent
In this paper, we define a new norm, a new concept of numerical range  and a new notion of numerical radius for operators on Hilbert $\mathscr{A}$-modules, where $\mathscr{A}$ is an abelian $C^*$-algebra. We investigate the above inequalities in the framework of Hilbert $\mathscr{A}$-modules.\\
\indent
Recall that a right pre-Hilbert $C^*$-module $\mathcal{E}$ over a $C^*$-algebra $\mathscr{A}$ (or a right pre Hilbert $\mathscr{A}$- module) is a linear space which is right $\mathscr{A}$-module equipped with an $\mathscr{A}$-valued inner product $\langle \cdot , \cdot \rangle:\mathcal{E}\times\mathcal{E}\rightarrow \mathscr{A}$ that satisfies the following properties:\\
$(i) \langle x,\alpha y+\beta z\rangle=\alpha\langle x,y\rangle+\beta\langle x,z\rangle$\\
$(ii)\langle x,ya\rangle=\langle x,y\rangle a$\\
$(iii)\langle x,y\rangle^*=\langle y,x\rangle$\\
$(iv)\langle x,x\rangle\geq0$; if $\langle x,x\rangle=0$ then $x=0$\\
 for each $x,y,z\in\mathcal{E},~~~ a\in \mathscr{A}$ and $\alpha,\beta\in\mathbb{C}$.\\
\indent
A pre Hilbert $\mathscr{A}$-module which is complete with respect to
the norm $\Vert x\Vert=\Vert\langle x,x \rangle\Vert^\frac{1}{2}$ is called a Hilbert $C^*$-module over $\mathscr{A}$, or a Hilbert $\mathscr{A}$-module.
Suppose that $\mathcal{E}$ and $\mathcal{F}$ are Hilbert $\mathcal{A}$-modules. We define $L(\mathcal{E},\mathcal{F})$ to be
the set of all maps $T:\mathcal{E}\to \mathcal{F}$ for
which there is a map $T^*:\mathcal{F}\to \mathcal{E}$
such that $\langle Tx,y \rangle=\langle x,T^*y
\rangle$ for all $x \in \mathcal{E}, y\in \mathcal{F}$. It is known
that $T$ must be a bounded $\mathcal{A}$-linear map (that
is, $T$ is bounded linear map and
$T(xa)=T(x)a$ for all $x \in \mathcal{E},a\in\mathcal{A}$).
 If $\mathcal{E}=\mathcal{F}$, then $L(\mathcal{E})$ is a $C^{*}$-algebra
together with the operator norm.

Suppose that $\mathscr{A}$ is an abelian $C^*$-algebra. Recall that a character $\varphi$ on $\mathscr{A}$
is a non-zero $*$-homomorphism $\varphi:\mathscr{A}\rightarrow\mathbb{C}$ such that $\|\varphi\|=1$.
 We denote the set of all characters on $\mathscr{A}$ by $\tau(\mathscr{A})$.

\section{Main results}
In the rest of the paper we assume that $\mathscr{A}$ is an abelian $C^*$-algebra. We start this section with the following definition.
\begin{definition}\label{1}
Let $T \in L(\mathcal{E}).$
\begin{equation}\label{def}
\interleave T\interleave:\equiv \sup \{ \varphi(|Tx|) : x \in \mathcal{E},\,\, \varphi\in\tau(\mathscr{A})\quad \& \quad\varphi(|x|)=1\},
 \end{equation}
where $|x|=\langle x,x\rangle^\frac{1}{2}$.
\end{definition}
First we show that $\interleave \cdot \interleave$ is a norm on $\mathcal{E}.$\\
If $T=0$, it is obvious that $\interleave T\interleave=0.$\\
If $\interleave T\interleave=0$, then  for every $\varphi\in\tau(\mathscr{A})$
and each $x\in\mathcal{E}$ such that $\varphi(|x|)=1$, we have $\varphi(|Tx|)=0$.
We want to show that $Tx=0$ for each $x \in\mathcal{E}.$\\
Fix $x\in\mathcal{E}$,
 
(i) If $\varphi(|x|)=0$, then by the Cauchy-Schwarz inequality we have
\[\varphi(\langle Tx,Tx\rangle)=\varphi(\langle T^*Tx,x\rangle)\leq\varphi(\langle T^*Tx,T^*Tx\rangle)^\frac{1}{2}\varphi(\langle x,x\rangle)^\frac{1}{2},
\]
thus $\varphi(|Tx|)=0$.

(ii)  If $\varphi(|x|)\neq 0$, then by taking
 $y=\frac{x}{\varphi(|x|)}$, we get $\varphi(|y|)=1$. 
 By definition $\ref{def}$, $\varphi(|Ty|)=0$ and so $\frac{1}{\varphi(|x|)}\varphi(|Tx|)=0$. Thus $\varphi(|Tx|)=0$.
Since for every $\varphi\in\tau(\mathscr{A})$, we have $\varphi(|Tx|)=0$. 
We conclude that $|Tx|=0$ for each $x \in\mathcal{E}.$ So $T=0$.

On the other hand $\mathscr{A}$ is an abelian $C^*$-algebra, then by \cite[Theorem 3.6]{Jia}, $|x+y|\leq |x|+|y|$ for each $x,y\in\mathcal{E}$.

Thus
  $$|T(x)+S(x)|\leq|T(x)|+|S(x)|$$
   for each $T,S\in L(\mathcal{E})$ and $x\in\mathcal{E}$.
Hence
\begin{align*}
\varphi(|(T+S)x|)=\varphi(|T(x)+S(x)|\leq \varphi(|T(x)|)+\varphi(|S(x)|)
\end{align*}
Now by taking the supremum over $x\in\mathcal{E}$ and $\varphi \in \tau(\mathscr{A})$
with $\varphi(|x|)=1$, we get $$\interleave T+S\interleave \leq \interleave T\interleave+\interleave S\interleave.$$
Clearly
$\interleave \alpha T\interleave=|\alpha|\interleave T\interleave$, for $\alpha\in \mathcal{C}.$
\begin{remark}
If $\mathcal{E}$ is Hilbert space, then
$$\varphi(|Tx|)=\varphi(\langle Tx,Tx\rangle^\frac{1}{2})=\varphi(\Vert Tx\Vert)=\Vert Tx\Vert\varphi(1)=\Vert Tx\Vert.$$
Similary $\varphi(|x|)=\Vert x\Vert$. Hence $\|T\|=\interleave T\interleave$.
\end{remark}

\begin{theorem}\label{norm}
If $\mathcal{E}$ is a Hilbert $\mathscr{A}$-module, then
$$\interleave T\interleave=\sup\{|\varphi(\langle Tx,y\rangle)|: x,y\in \mathcal{E},\,\ \varphi\in \tau(\mathscr{A})\quad \& \quad\varphi(|x|)=\varphi(|y|)=1\}.$$
\end{theorem}
\begin{proof}
Let $\beta=\sup\{|\varphi(\langle Tx,y\rangle)| :x,y\in\mathcal{E},\,\ \varphi\in \tau(\mathscr{A})\quad \& \quad \varphi(|x|)=\varphi(|y|)=1\}$. It is sufficient to prove that $\interleave T\interleave=\beta$.\\
If $\varphi\in\tau(\mathscr{A})$ and $x,y\in \mathcal{E},$ such that $\varphi(|x|)=\varphi(|y|)=1$, then by using  the Cauchy-Schwarz inequality, we get
\begin{align*}
|\varphi(\langle Tx,y\rangle)|&\leq\varphi(\langle Tx,Tx\rangle)^\frac{1}{2} \varphi(\langle y,y\rangle)^\frac{1}{2}\cr
&=(\varphi(|Tx|)^2)^\frac{1}{2}(\varphi(|y|)^2)^\frac{1}{2}\cr
&=\varphi(|Tx|)\cr
&\leq \interleave T\interleave.
\end{align*}
 Hence $\beta\leq\interleave T\interleave.$
\\
For every $\varphi\in\tau(\mathscr{A})$ and $x\in \mathcal{E}$, with $\varphi(|x|)=1$, we have
\[\varphi(|Tx|)^2=\varphi(|Tx|^2)=\varphi(\langle Tx,Tx \rangle)=\varphi(|Tx|)\varphi\left(\langle Tx, \frac{Tx}{\varphi(|Tx|)}\rangle\right),\]
where we assume that $\varphi(|Tx|)\neq 0.$
Thus
\begin{align*}
\varphi(|Tx|)&=\varphi \left(\langle Tx,\frac{Tx}{\varphi(|Tx|)}\rangle\right)\cr
&\leq \sup\{|\varphi(\langle Tx,y\rangle)|: x,y\in\mathcal{E},\,\ \varphi\in \tau(\mathscr{A})\quad \& \quad\varphi(|x|)=\varphi(|y|)=1\}.
\end{align*}
Therefore, $\varphi(|Tx|) \leq\beta $. Hence
$\interleave T\interleave \leq\beta$.
\end{proof}
\begin{theorem}\label{3}
If $T\in L(\mathcal{E})$ is self-adjoint, then
\begin{equation}
\interleave T\interleave=\sup\{|\varphi(\langle Tx,x\rangle)|: x\in \mathcal{E},\,\ \varphi\in\tau(\mathscr{A})\quad \& \quad \varphi(|x|)=1\}.
\end{equation}
\end{theorem}
\begin{proof}
 Let $M=\sup\{|\varphi(\langle Tx,x\rangle)|: x\in\mathcal{E},\,\ \varphi\in\tau(\mathscr{A})\quad \&\quad\varphi(|x|)=1\}$.
 If $\varphi\in \tau(\mathscr{A})$ and $T\in L(\mathcal{E})$ is self-adjoint, then by using the Cauchy-Schwartz inequality
\begin{align*}
|\varphi(\langle Tx,x\rangle)|&\leq\varphi(\langle Tx,Tx\rangle)^\frac{1}{2} \varphi(\langle x,x\rangle)^\frac{1}{2}\cr
&=(\varphi(|Tx|)^2)^\frac{1}{2}(\varphi(|x|)^2)^\frac{1}{2}.
\end{align*}
If $\varphi(|x|)=1$, then
\begin{align}\label{quick}
|\varphi (\langle Tx,x \rangle)| & \leq \varphi(|Tx|)\cr
|\varphi(\langle Tx,x\rangle)| &\leq \sup\{\varphi(|Tx|): x\in\mathcal{E},\,\ \varphi\in \tau(\mathscr{A})\quad\&\quad\varphi(|x|)=1\}\cr
|\varphi( \langle Tx,x\rangle)| &\leq \interleave T\interleave.
\end{align}
By taking supremum over $\varphi(|x|)=1$, we get
$$ M\leq\interleave T\interleave.$$
Conversely, let $\varphi\in \tau(\mathscr{A})$ and $x,y\in \mathcal{E}$. Then
\begin{eqnarray*}
\varphi(\langle T(x+y),x+y\rangle)- \varphi(\langle T(x-y),x-y\rangle)
=4\varphi({\rm Re} \langle Tx,y\rangle).
\end{eqnarray*}
Hence
\begin{align*}
 \vert\varphi({\rm Re}\langle Tx,y\rangle)\vert &=\frac{1}{4}\vert\varphi(\langle T(x+y),x+y\rangle) - \varphi(\langle T(x-y),x-y\rangle)\vert\\
&\leq \frac{1}{4}\vert\varphi(\langle T(x+y),x+y\rangle)\vert+\frac{1}{4}\vert\varphi (\langle T(x-y),x-y\rangle)\vert\\
&=\frac{1}{4}\varphi (|x+y|^2)\left|\varphi \left(\langle\frac{T(x+y)}{\varphi(|x+y|)},\frac{x+y}{\varphi (|x+y|)}\rangle\right)\right|\\
&\quad+\frac{1}{4}\varphi (|x-y|^2)\left|\varphi \left(\langle \frac{T(x-y)}{\varphi (|x-y|)},\frac{x-y}{\varphi (|x-y|)}\rangle\right)\right|.
\end{align*}
since $\varphi(\left|\frac{x+y}{\varphi(|x+y|)}\right|)=\frac{\varphi(|x+y|)}{\varphi(|x+y|)}=1$, we obtain

$\left|\varphi(\langle \frac{T(x+y)}{\varphi(|x+y|)},\frac{x+y}{\varphi(|x+y|)}\rangle)\right|\leq M$  and $\left|\varphi(\langle \frac{T(x-y)}{\varphi (|x-y|)},\frac{x-y}{\varphi(|x-y|)}\rangle)\right|\leq M,$

whence
\begin{align*}
|\varphi({\rm Re}\langle Tx,y\rangle)|&\leq\frac{1}{4}M(\varphi( |x+y|^2)+\varphi (|x-y|^2))
=\frac{1}{4}M\varphi(|x+y|^2+|x-y|^2)\\
&=\frac{1}{4}M\varphi(2|x|^2+2|y|^2)
=\frac{1}{2}M\varphi(|x|^2+|y|^2).
\end{align*}

If $y=\frac{Tx}{\varphi(|Tx|)}$ and $\varphi(|x|)=1$, then
\begin{align*}
\big|\varphi({\rm Re}\langle Tx,\frac{Tx}{\varphi |Tx|}\rangle)\big|
 &\leq \frac{M}{2} \varphi \left(|x|^2+\big|\frac{Tx}{\varphi( |Tx|)}\big|^2\right)\\
 &=\frac{M}{2}\varphi\left( ( |x|^2+\frac{|Tx|^2}{\varphi (|Tx|^2)}\right)\\
 &=\frac{M}{2}\left(\varphi (|x|^2) +\frac{\varphi (|Tx|^2)}{\varphi (|Tx|^2)}\right)\\
 &=M.
\end{align*}
Hence
\begin{align*}
 \left|\frac{1}{\varphi (|Tx|)}\varphi \left({\rm Re}\langle Tx,Tx\rangle\right)\right|
&=\left|\frac{1}{\varphi (|Tx|)} {\rm Re}(\varphi (|Tx|^2))\right|\\
&=\left|\frac{1}{\varphi(|Tx|)}\varphi (|Tx|^2)\right|\\
&= \varphi (|Tx|) \leq M.
\end{align*}
\end{proof}
\section{Numerical range and Numerical radius}
In this section, we define the numerical range and numerical radius for operators on  $L(\mathcal{E})$,
according to the definition of $\interleave  \cdot \interleave$ on $L(\mathcal{E})$.
\begin{definition}
Let $T\in L(\mathcal{E})$. Then the numerical range of $T$ is defined by
\begin{equation}
W_{o}(T)=\{ \varphi(\langle Tx,x\rangle): x\in \mathcal{E},\,\ \varphi \in\tau(\mathscr{A})\quad \&\quad\varphi(|x|)=1\}.
\end{equation}
\end{definition}
The next result represent some of the basic properties for the numerical range.
\begin{theorem}
If $T,S\in L(\mathcal{E})$, then
\begin{itemize}
\item[(i)] $W_{o}(T^*)=\overline{W_{o}(T)}$, where $\overline{W_{o}(T)}$ is conjugate of $W_{o}(T)$.
\item[(ii)]If $\alpha,\beta\in\mathbb{C}$, then $W_{o}(\alpha T+\beta I_\mathcal{E})=\alpha W_{o}(T)+\beta$.
\item[(iii)]If $U\in L(\mathcal{E})$ is unitary, then $W_{o}(UTU^*)=W_{o}(T)$.
\item[(iv)] $W_{o}(T)\subseteq \mathbb{R}$ if and only if $T$ is self-adjoint.
\item[(v)] $W_{o}(T+S)\subseteq W_{o}(T)+W_{o}(S)$.
\end{itemize}
\end{theorem}
\begin{proof}

(i)
\begin{align*}
W_{o}(T^*)&=\{ \varphi(\langle T^*x,x\rangle): x\in\mathcal{E},\,\ \varphi\in \tau(\mathscr{A})\quad \&\quad\varphi(|x|)=1\}\cr
&=\{ \varphi(\langle x,Tx\rangle): x\in\mathcal{E},\,\ \varphi\in \tau(\mathscr{A})\quad\&\quad\varphi(|x|)=1\}\cr
&=\{\overline{\varphi(\langle Tx,x \rangle)}:x\in\mathcal{E},\,\ \varphi\in(\tau(\mathscr{A})\quad\&\quad\varphi(|x|)=1\}\cr
&=\{ \bar{\lambda}:\lambda \in W_{o}(T)\}=\overline{W_{o}(T)}.
\end{align*}
(ii) It is clear.\\
(iii) Since $\varphi(|x|)=\varphi(|Ux|)=1,$ we have
\begin{align*}
W_{o}(U^*TU) &=\{\varphi(\langle U^*TUx,x\rangle): x\in\mathcal{E},\,\ \varphi\in \tau(\mathscr{A})\quad\&\quad\varphi(|x|)=1\}\cr
&=\{\varphi(\langle TUx,Ux\rangle): x\in\mathcal{E},\,\ \varphi\in \tau(\mathscr{A})\quad\&\quad\varphi(|x|)=1\}\cr
&=\{\varphi(\langle Ty,y\rangle): y\in\mathcal{E},\,\ \varphi\in \tau(\mathscr{A})\quad\&\quad\varphi(|y|)=1\}\cr
&=W_{o}(T).
\end{align*}
(iv) If $T\in L(\mathcal{E})$ is self-adjoint, then\
\[\varphi(\langle Tx,x\rangle)=\varphi(\langle x,Tx\rangle)=\varphi(\langle Tx,x\rangle)^*=\overline{\varphi(\langle Tx,x\rangle)},\]
which is equivalent to $\varphi(\langle Tx,x\rangle) \in\mathbb{R}$.\\
Conversely, if  $\varphi (\langle Tx,x\rangle) \in\mathbb{R}$, then $\varphi(\langle Tx,x\rangle)=\varphi(\langle T^*x,x\rangle)$
i.e. $\varphi(\langle (T-T^*)x,x\rangle)=0 $.
 Hence $\langle (T-T^*)x,x\rangle=0$ for every $x\in\mathcal{E}.$ Thus $T=T^*$.\\
 (v) Since
\begin{align*}
W_{o}(T+S)&=\{\varphi(\langle (T+S)x,x\rangle) : x\in\mathcal{E},\,\ \varphi\in \tau(\mathscr{A})\quad \&\quad\varphi(|x|)=1\}\cr
&=\{\varphi(\langle Tx,x\rangle) +\varphi(\langle Sx,x\rangle) : x\in\mathcal{E},\,\ \varphi\in \tau(\mathscr{A})\quad\&\quad\varphi(|x|)=1\}, \end{align*}
and $\varphi(\langle Tx,x\rangle)\in W_{o}(T)$, $\varphi(\langle Sx,x\rangle)\in W_{o}(S)$, we arrive at the result.
 \end{proof}
\begin{definition}\label{5}
Let $T\in L(\mathcal{E})$. Then the numerical radius of $T$ is defined by
\begin{equation}
\omega_{o}(T)=\sup\{|\varphi(\langle Tx,x\rangle)|: x\in\mathcal{E},\,\ \varphi\in \tau(\mathscr{A})\quad \&\quad\varphi(|x|)=1\}.
\end{equation}
\end{definition}

It is easy to show that $\omega_{o}(.)$ is a norm on $ L(\mathcal{E})$.
\begin{lemma}
If $\mathcal{E}$ is a Hilbert $\mathscr{A}$-module, then for every $\varphi\in\tau(\mathscr{A})$, $x\in\mathcal{E},$
\begin{equation}
\varphi(|\langle Tx,x\rangle|)\leq(\varphi(|x|^2)\omega_{o}(T)
\end{equation}
\end{lemma}
\begin{proof}
For each $\varphi\in\tau(\mathscr{A})$ and $x\in\mathcal{E}$ we have
$$\frac{1}{\varphi(|x|^2)}\big|\varphi(\langle Tx,x\rangle)\big|=\big|\varphi(\langle \frac{Tx}{\varphi(|x|)},\frac{x}{\varphi(|x|)}\rangle)\big|$$
hence
$$\frac{1}{\varphi(|x|^2)}|\varphi(\langle Tx,x\rangle)|\leq\omega_{o}(T)\Longrightarrow|\varphi(\langle Tx,x\rangle)|\leq (\varphi(|x|^2)\omega_{o}(T),$$
since $\varphi(|\frac{x}{\varphi(|x|)}|)=1$.
\end{proof}

In the next result, we show that $\interleave \cdot \interleave$ and $\omega_{o}(\cdot)$ are equivalent.
\begin{theorem}
If $T\in L(\mathcal{E})$, then
\begin{equation}
\omega_{o}(T)\leq\interleave T\interleave\leq 2\omega_{o}(T).
\end{equation}
\end{theorem}
\begin{proof}
 For every $\varphi\in\tau(\mathscr{A})$ and  $x\in\mathcal{E}$  such that $\varphi(|x|)=1$, by Theorem $(\ref{norm})$, we have
 $$|\varphi(\langle Tx,x\rangle)|\leq\interleave T\interleave,$$
 By getting supremum, we obtain
\begin{equation}
\omega_{o}(T)\leq\interleave T\interleave.
\end{equation}
Fix $x,y\in\mathcal{E}$ and $\varphi\in\tau(\mathscr{A}).$
\begin{align*}
4|\varphi(\langle Tx,y\rangle)|&=|\varphi(\langle T(x+y),x+y\rangle-\langle T(x-y),x-y\rangle\cr
 &\quad+ i\langle T(x+iy),x+iy\rangle- i\langle T(x-iy),x-iy\rangle)|\cr
&\leq|\varphi(\langle T(x+y),x+y\rangle)| +|\varphi (\langle(x-y),x-y\rangle)|\cr
&\quad+|\varphi(\langle T(x+iy),x+iy\rangle)|+|\varphi(\langle T(x-iy,x-iy\rangle)|,\cr
\end{align*}
Thus
\begin{align*}
 |\varphi(\langle Tx,y\rangle)|&\leq\frac{1}{4}(\varphi(|x+y|^2)\omega_{o}(T)
+\varphi(|x-y|^2)\omega_{o}(T)\quad( \ref{5})\quad\cr
&\quad+\varphi(|x+iy|^2)\omega_{o}(T)+\varphi(|x-iy|^2)\omega_{o}(T))\cr
&=\frac{1}{4}\omega_{o}(T)(\varphi(| x+y|^2)+\varphi(| x-y|^2)+\varphi(| x+iy|^2)+\varphi(| x-iy|^2))\cr
&=\frac{1}{4}(\omega_{o}(T)\varphi(2|x|^2+2|y|^2+2|x|^2+2|iy|^2))\cr
&=\omega_{o}(T)\varphi(|x|^2+|y|^2).
\end{align*}
If $\varphi(|x|)=\varphi(|y|)=1$, then
$$\big|\varphi(\langle Tx,y\rangle)\big|\leq 2\omega_{o}(T).$$
Hence
\begin{equation}
\interleave T\interleave\leq 2\omega_{o}(T).
\end{equation}
\end{proof}
We use some similar strategies as in \cite[Theorem 1]{Kit} to prove the next  result.
\begin{theorem}
If $T\in L(\mathcal{E})$, then
\begin{equation}\label{3.7}
\frac{1}{4}\interleave T^*T+TT^*\interleave\leq(\omega_{o}(T))^2\leq\frac{1}{2}\interleave T^*T+TT^*\interleave.
\end{equation}
\end{theorem}
\begin{proof}
Let $T=M+iN$, where
 $M$ and $N$ are self-adjoint and $T^*T+TT^*=2(M^2+N^2)$. Let $x\in\mathcal{E}$. From convexity of the function $f(t)=t^2$, we have
\begin{align*}
|\varphi\langle Tx,x\rangle|^2 &=|\varphi\langle (M+iN)x,x\rangle|^2 =|\varphi\langle Mx,x\rangle +i\varphi\langle Nx,x\rangle|^2\cr
&=(\varphi\langle Mx,x\rangle)^2+(\varphi\langle Nx,x\rangle)^2\cr
&\geq\frac{1}{2}(|\varphi\langle Mx,x\rangle|+|\varphi\langle Nx,x\rangle|)^2\cr
&\geq\frac{1}{2}|\varphi\langle Mx,x\rangle\pm\varphi\langle Nx,x\rangle|^2\cr 
&=\frac{1}{2}|\varphi(\langle Mx,x\rangle\pm\langle Nx,x\rangle)|^2\cr
&=\frac{1}{2}|\varphi\langle M\pm N)x,x\rangle|^2.
\end{align*}
Hence 
\begin{align*}
(\omega_{o}(T))^2 &=\sup\lbrace|\varphi\langle Tx,x\rangle|^2:  x \in \mathcal{E},\,\, \varphi\in\tau(\mathscr{A})\quad \& \quad\varphi(|x|)=1\}\\
&\geq\frac{1}{2}\sup\lbrace|\varphi \langle(M\pm N)x,x\rangle|^2:  x \in \mathcal{E},\,\, \varphi\in\tau(\mathscr{A})\quad \& \quad\varphi(|x|)=1\}\\
&=\frac{1}{2}\interleave(M\pm N)^2\interleave.
\end{align*}
So
\begin{align*}
2(\omega_{o}(T))^2&\geq \frac{1}{2}\interleave(M+N)^2\interleave+\frac{1}{2}\interleave(M-N)^2\interleave\\
&\geq \frac{1}{2}\interleave (M+N)^2+(M-N)^2\interleave\\
&=\interleave M^2+N^2\interleave\\
&=\frac{1}{2}\interleave T^*T+TT^*\interleave.
\end{align*}
Therefore
$$(\omega_{o}(T))^2\geq\frac{1}{4}\interleave T^*T+TT^*\interleave.$$
 To prove the right hand inequality,
let $\varphi\in\tau(\mathcal{A})$ and $x\in\mathcal{E}$ such that $\varphi|x|=1$.
From the Cauchy--Schwartz inequality, we have 
\begin{align*}
|\varphi\langle Tx,x\rangle|^2 &=(\varphi\langle Mx,x\rangle)^2+(\varphi\langle Nx,x\rangle)^2\\
&\leq\varphi\langle Mx,Mx\rangle\varphi\langle x,x\rangle+\varphi\langle Nx,Nx\rangle\rangle\varphi\langle x,x\rangle\\
&\leq\varphi\langle Mx,Mx\rangle\varphi|x|^2+\varphi\langle Nx,Nx\rangle\rangle\varphi|x|^2\\
&=\varphi\langle Mx,Mx\rangle+\varphi\langle Nx,Nx\rangle\\
&=\varphi\langle M^2x,x\rangle+\varphi\langle N^2x,x\rangle\\
&=\varphi\langle (M^2+N^2)x,x\rangle.
\end{align*}
Hence 
\begin{align*}
(\omega_{o}(T))^2&=\sup\lbrace|\varphi\langle Tx,x\rangle|^2:  x \in \mathcal{E},\,\, \varphi\in\tau(\mathscr{A})\quad \& \quad\varphi(|x|)=1\}\\
&\leq\sup\lbrace|\varphi\langle (M^2+N^2)x,x\rangle|:  x \in \mathcal{E},\,\, \varphi\in\tau(\mathscr{A})\quad \& \quad\varphi(|x|)=1\}\\    
&=\interleave M^2+N^2\interleave\\
&=\frac{1}{2}\interleave T^*T+TT^*\interleave,
\end{align*}
which complete the proof.
\end{proof}
\begin{example}
Let X be a compact Hausdorff space and $\mathcal{E}=\mathscr{A}=C(X)$. Then C(X) is a Hilbert C(X)-module, such that $\langle f,g \rangle=\bar{f}g$ for each $f,g\in C(X)$. Let $\varphi\in\tau(C(X))$. Then by Theorem 2.1.15 in \cite{Mur}, there exists $x\in\mathit{X}$ such that $\varphi=\varphi_{x}$, where $\varphi_{x}(f)=f(x)$ for each $f\in C(X)$.\\ Thus

\begin{align*}
\interleave T\interleave&=\sup\{ \varphi(|Tf|):  f\in \mathscr{A},\,\ \varphi\in\tau(\mathscr{A})\quad\& \quad\varphi(|f|)=1\}\cr
&=\sup\{ |Tf|(x) : f\in \mathscr{A},\,\ \varphi\in\tau(\mathscr{A})\quad \& \quad\vert f|(x)=1\}\cr
&=\sup\{ |Tf(x)|:  f\in \mathscr{A},\,\ \varphi\in\tau(\mathscr{A})\quad \& \quad\vert f(x)|=1\}.
\end{align*}
Also
\begin{align*}
&\sup\{|\varphi(\langle Tf,g\rangle)|: f,g\in\mathscr{A},\,\ \varphi\in\tau(\mathscr{A})\quad  \&\quad\varphi(|f|)=\varphi(|g|)=1\}\cr
&\quad=\sup\{ |\varphi((\overline{Tf})g)|: f,g\in\mathscr{A},\,\ \varphi\in\tau(\mathscr{A})\quad \&\quad\varphi(|f|)=\varphi(|g|) =1\}\cr
&\quad=\sup\{|\overline{Tf}g(x)|:  x\in\mathcal{E},\,\  f,g\in\mathscr{A},\,\ \varphi\in\tau(\mathscr{A})\quad\& \quad|f|(x)=|g|(x)=1\}\cr
&\quad=\sup\{|\overline{Tf}(x)||g(x)|: x\in\mathcal{E},\,\  f,g\in\mathscr{A},\,\ \varphi\in\tau(\mathscr{A})\quad\& \quad|f|(x)=|g|(x)=1\}\cr
&\quad=\sup\{|\overline{Tf}(x)|: x\in\mathcal{E},\,\  f\in\mathscr{A},\,\ \varphi\in\tau(\mathscr{A})\quad \&\quad|f(x)|=1\}\cr
&\quad=\sup\{|Tf(x)|:  x\in\mathcal{E},\,\  f\in\mathscr{A},\,\ \varphi\in\tau(\mathscr{A})\quad \&\quad|f(x)|=1\}.\cr
\end{align*}
Hence
\begin{align*}
\interleave T\interleave=&\sup\{| \varphi\langle Tf,g\rangle|: f,g\in\mathscr{A},\,\ \varphi\in\tau(\mathscr{A}) \quad \&\quad\varphi(|f|)=\varphi(|g|)=1\}\cr
=&\sup\{|Tf(x)|: x\in\mathcal{E},\,\  f,g\in\mathscr{A},\,\ \varphi\in\tau(\mathscr{A})\quad \&\quad|f(x)|=1\}.
\end{align*}

If $T$ is self-adjoint, then
$$\interleave T\interleave=\sup \{|\varphi(\langle Tf,f \rangle)|: f\in\mathscr{A},\,\ \varphi\in\tau (\mathscr{A})   \quad \& \quad\varphi(|f|)=1\},$$
whence
$$\omega_{o}(T)=\interleave T\interleave.$$
\end{example}

\end{document}